\documentclass[a4paper,11pt,oneside,final]{amsart}
\usepackage[utf8]{inputenc}
\usepackage[T1]{fontenc}
\usepackage[english]{babel}
\usepackage{amsmath,amssymb,amsfonts,amsthm}
\usepackage[inline]{enumitem}
\usepackage{mathrsfs}
\usepackage[all]{xy}
\usepackage{tikz}
\usepackage[final]{hyperref}

\hypersetup{%
    colorlinks=true,%
    linkcolor=blue,%
    citecolor=blue,%
    pdftitle    = {Non-inner amenability of the Thompson groups T and V},%
    pdfauthor   = {Kristian Knudsen Olesen},%
    pdfsubject  = {Mathematics},%
    pdfkeywords = {Operator algebra, Thompson groups, inner amenability,%
                   amenibility, C*-simplicity, unique trace property},%
    }

\theoremstyle{plain}
    \newtheorem{theorem}{Theorem}[section]
    \newtheorem{proposition}[theorem]{Proposition}
    \newtheorem{corollary}[theorem]{Corollary}
    \newtheorem{lemma}[theorem]{Lemma}

\theoremstyle{remark}
    \newtheorem{remark}[theorem]{Remark}

\theoremstyle{definition}
    \newtheorem{definition}[theorem]{Definition}

\newcommand{\I}{\mathbf{1}}
\newcommand{\LL}{\mathrm{L}}

\DeclareMathOperator{\SL}{\mathrm{SL}}
\DeclareMathOperator{\PSL}{\mathrm{PSL}}
\DeclareMathOperator{\PPSL}{\mathrm{PPSL}}
\DeclareMathOperator{\cconv}{\overline{conv}}

\title{Non-inner amenability of the Thompson groups $T$ and $V$}
\date{\today}

\author{Uffe Haagerup}
\address{%
    Department of Mathematics and Computer Science\\
    University of Southern Denmark\\
    Campusvej~55\\
    DK-5230 Odense~M%
}

\author{Kristian Knudsen Olesen}
\address{%
    Department of Mathematical Sciences\\
    University of Copenhagen\\
    Universitetsparken~5\\
    DK-2100 Copenhagen~Ø%
}
\email{olesen@math.ku.dk}

\thanks{Both authors are supported by the ERC Advanced Grant no.\ OAFPG 247321
and the Danish National Research Foundation through the Centre for Symmetry
and Deformation (DNRF92). The first author is also supported by the Danish
Natural Science Research Council.}

\begin{document}

\begin{abstract}
    In this paper we prove that the Thompson groups~$T$ and~$V$ are not inner
    amenable.  In particular, their group von Neumann algebras do not have
    property~$\Gamma$.  Moreover, we prove that if the reduced group
    $C^\ast$\mbox-algebra of~$T$ is simple, then the Thompson group~$F$ is
    non-amenable.  Furthermore, we give a few new equivalent characterizations
    of amenability of~$F$.
    \end{abstract}

\maketitle

\section{Introduction}
                                                             \label{sec:intro}

The Thompson groups $F$,~$T$ and~$V$ where introduced by Richard Thompson in
1965.  They are countable discrete groups of piecewise linear bijections of
the half open interval~$[0,1)$ onto itself and satisfy $F \subseteq T
\subseteq V$.  We refer to the paper~\cite{CFP96} by Cannon, Floyd and Parry
for a detailed introduction to the subject.  It is well-known that the
Thompson groups~$T$ and~$V$ are non-amenable, but it is a major open problem
to decide whether or not $F$ is amenable.  All three groups are ICC, that is,
all conjugacy classes of non-trivial elements are infinite, and hence their
von Neumann algebras $\LL F$, $\LL T$ and $\LL V$ are all factors of type
II$_1$.  Jolissaint~\cite{Jol97} proved that $F$ is inner amenable in the
sense of Effros~\cite{Eff}.  He later strengthened this result by proving that
$\LL F$ is a McDuff factor; see~\cite{Jol98}.  In particular, this means that
$\LL F$ has property~$\Gamma$, and hence, by a result of Effros~\cite{Eff},
$F$ is inner amenable.  The first and main result of this paper is that $T$
and~$V$ are not inner amenable, and hence, by the same result of Effros, $\LL
T$ and~$\LL V$ do not have property~$\Gamma$.  Therefore, they are also not
McDuff factors.  The fact that $F$ is inner amenable was proved in a different
way by Ceccherini-Silberstein and Scarabotti~\cite{CSS01}, a few years later.

Another result of this paper connects amenability of~$F$ with simplicity of
the reduced group $C^\ast$\mbox-algebra of~$T$.  More precisely, we prove that
if the reduced group $C^\ast$\mbox-algebra $C_r^\ast(T)$ is simple, then $F$
is non-amenable.  Shortly after the results of this paper where announced by
the first named author at the Fields Institute in Toronto in March~2014, this
result was re-obtained by Breuillard, Kalantar, Kennedy and
Ozawa~\cite{BKKO14}, using completely different methods.  Very recently, this
result was also re-proved by Le~Boudec and Matte Bon~\cite{LBMB16}.  They also
showed the converse of this statement.  This result is not the only connection
between amenability of~$F$ and simplicity of its reduced group
$C^\ast$-algebra.  Indeed, it has been well-known for some time that $F$ is
amenable if and only if its reduced group $C^\ast$\mbox-algebra has a unique
tracial state, and it has been a long-standing open problem set forth by de la
Harpe (see, for example,~\cite{dlH07}) to decide whether the reduced group
$C^\ast$\mbox-algebra of a group is simple if and only if it has a unique
tracial state.  During the last few years tremendous progress has been made on
the topic of $C^\ast$-simple groups (that is, groups whose reduced group
$C^\ast$-algebra is simple) and groups with the unique trace property (that
is, groups whose reduced group $C^\ast$-algebra has a unique tracial state).
This new development began when Kalantar and Kennedy~\cite{KK14} gave new
characterizations of $C^\ast$\mbox-simple groups in terms of boundary actions.
Immediately after, Breuillard, Kalantar, Kennedy and Ozawa~\cite{BKKO14}
proved several striking results, including that $C^\ast$-simplicity implies
the unique trace property.  Very recently, Le~Boudec,~\cite{Bou15}, gave
counterexamples to the remaining implication of the long-standing open problem
of de la Harpe by providing examples of groups which are not $C^\ast$-simple,
but have the unique trace property.  This settled completely the relationship
between $C^\ast$-simplicity and the unique trace property.  Further examples
of non-$C^\ast$-simple groups with the unique trace property were provided by
Ivanov and Omland~\cite{IO16}.

In addition to the results mentioned thus far, we investigate the
$C^\ast$\mbox-al\-ge\-bras generated by the images of~$F$,~$T$ and~$V$ via a
representation discovered by Nekrashevych~\cite{Nek04}. More precisely, we
prove that these $C^\ast$-algebras are distinct, and that that one generated
by the image of~$V$ is the Cuntz algebra~$\mathcal O_2$.  Furthermore, we give
new characterizations of amenability of~$F$ in terms of the size of certain
explicit ideals in the reduced group $C^\ast$-algebras of~$F$ and $T$, as well
as in terms of whether certain convex hulls contain zero.

We end the introduction by giving the definition of the Thompson groups.  The
Thompson group~$V$ is the set of all piecewise linear bijections of $[0,1)$
which are right continuous, have finitely many points of
non-differentiability, all being dyadic rationals, and have a derivative which
is a power of~$2$ at each point of differentiability.  It is easy to see that
these maps fix the set of dyadic rationals in~$[0,1)$.  The Thompson group $T$
consists of those elements of~$V$ which define a homeomorphism of $[0,1)$,
when identified with $\mathbb R / \mathbb Z$ in the standard way.
Equivalently, $T$ is the set of elements in~$V$ which have at most one point
of discontinuity with respect to the usual topology on~$[0,1)$.  Finally, the
Thompson group $F$  consists of those elements of~$V$ which are, in fact,
homeomorphisms of $[0,1)$, or, equivalently, the set of element $g \in T$
satisfying $g(0) = 0$.  All three groups are finitely generated and also
finitely presented in the following generators: $F$ is generated by~$A$
and~$B$; $T$ is generated by $A$,~$B$ and~$C$; and $V$ is generated by
$A$,~$B$, $C$ and~$\pi_0$, where $A$,~$B$, $C$ and~$\pi_0$ are defined
in~\cite{CFP96}.  Besides these standard generators, we will also consider the
following element~$D$ of~$T$, given by $D(x) = x + \frac 34$, for $x \in
[0,\frac 14)$, and $D(x) = x - \frac 14$ for $x \in [\frac 14,1)$.  Note that
$C^3 = D^4 = \I$.  Moreover, the elements~$C$ and~$D$ also generate~$T$ as a
group, as it is easily seen that $A = D^2 C^2$ and $B = C^2 D A$.

\section{The isomorphism between \texorpdfstring{$T$}{T} and
\texorpdfstring{$\PPSL(2, \mathbb Z)$}{PPSL(2,Z)}}
                                                               \label{sec:iso}

In order to prove that $T$ and~$V$ are not inner amenable, we shall use that
$T$ is isomorphic to a group of piecewise fractional linear transformations.
This fact was originally proved by Thurston, as explained
in~\cite[\S7]{CFP96}.  However, Thurston realized $T$ as M\"{o}bius
transformations of the interval~$[0,1]$, whereas we will use a slightly
modified version of Thurston's result, given by Imbert
in~\cite[Theorem~1.1]{Imb97}, where $T$ is realized as M\"{o}bius
transformations of~$\mathbb R \cup \{\infty\}$ instead.  The construction is
essentially the same, and we shall explain the difference in
Remark~\ref{questionmark} below.  A description of the isomorphism is also
given by Fossas in~\cite{Fos11}.  The second named author would like to thank
Vlad Sergiescu for several fruitful discussions and his help in sorting out
the origin of the particular version of Thurston’s result we are presenting
here.

Recall that the group $\PSL(2, \mathbb Z) = \SL(2, \mathbb Z) / \{\pm\I\}$
acts in a natural way on $\mathbb R \cup \nobreak \{\infty\}$ via M\"{o}bius
transformations, that is, for $x \in \mathbb R \cup \nobreak \{\infty\}$,
\begin{equation}
                                                                \label{mobius}
    g(x) = \frac{ax+b}{cx+d}, \qquad
    \text{when} \quad g = \begin{bmatrix} a & b \\ c & d \end{bmatrix}
    \{\pm\I\} \in \PSL(2, \mathbb Z).
    \end{equation}
This action is faithful, in the sense that only the neutral element acts
trivially.  Recall also that the elements of $\PSL(2,\mathbb Z)$ act by
homeomorphisms in an orientation-preserving way.

\begin{definition}
    We denote by $\PPSL(2, \mathbb Z)$ the group of homeomorphisms of~$\mathbb
    R \cup \nobreak \{\infty\}$ which are piecewise in $\PSL(2, \mathbb Z)$,
    that is, piecewise of the form~\eqref{mobius}, and which have only
    finitely many breakpoints, all them being in~$\mathbb Q \cup \nobreak
    \{\infty\}$.
    \end{definition}

In the above definition, a breakpoint\label{def:breakpoint} of an element $g
\in \PPSL(2, \mathbb Z)$ should be understood as a point $x \in \mathbb R \cup
\nobreak \{\infty\}$ for which there is no open neighbourhood~$U$ so that
$g$~acts on~$U$ as an element of~$\PSL(2, \mathbb Z)$.

\begin{theorem}[Thurston]
                                                          \label{thm:thurston}
    There exists a homeomorphism $\phi$ of $\mathbb R \cup \nobreak
    \{\infty\}$ onto $\mathbb R / \mathbb Z = S^1$ such that the map $\Phi
    \colon \PPSL(2,\mathbb Z) \to T$ given by
    \begin{equation*}
        g \mapsto \phi \circ g \circ \phi^{-1}
        \end{equation*}
    is an isomorphism.
    \end{theorem}

As mentioned, Thurston's construction of the homeomorphism,
see~\cite[\S7]{CFP96}, is different than the one in the theorem above, which
is briefly explained in~\cite{Imb97}.  More details are given by Fossas in the
proof of Theorem~2.2 in~\cite{Fos11}.  More precisely, one first defines an
order preserving homeomorphism~$\psi$ of $[-\infty,\infty]$ onto $[-\frac 12,
\frac 12]$, and next obtains $\phi \colon \mathbb R\cup \{\infty\} \to \mathbb
R/ \mathbb Z$ from~$\psi$ by identifying the endpoints of each of the two
intervals, that is, $\phi(x) = \psi(x) + \mathbb Z$, for all $x \in \mathbb
R$, and $\phi(\infty) = \frac 12 + \mathbb Z$.  The map~$\psi$ is denoted
by~``$?$'' in~\cite{Fos11}.  It is closely related to (but different from) the
Minkowski question mark function, as explained in Remark~\ref{questionmark}
below.  To construct $\psi$, one writes each element in $\mathbb Q \cup
\nobreak \{\pm \infty\}$ as a reduced fraction $\frac  pq$ with $p,q \in
\mathbb Z$ and $q \geq 0$, using the convention that
\begin{center}
    $-\infty = \frac {-1} 0$ \qquad and \qquad $\infty = \frac 10$.
    \end{center}
Two reduced fractions $\frac pq$ and $\frac rs$ (possibly $\frac {-1}0$ and
$\frac 10$) are called consecutive Farey numbers if $|ps - rq| = 1$.  Given
two consecutive Farey numbers one defines
\begin{equation*}
    \frac pq \oplus \frac rs = \frac {p+r} {q+s}.
    \end{equation*}
Let now $\psi$ be defined on $\mathbb Q \cup \nobreak \{\pm \infty\}$ by
letting $\psi(-\infty) = -\frac 12$, $\psi(0) = 0$, $\psi(\infty) = \frac 12$,
and then recursively let $\psi(\frac pq \oplus \frac rs) = \frac 12\big(
\psi(\frac pq) + \psi(\frac rs) \big)$, whenever $\frac pq$ and $\frac rs$ are
consecutive Farey numbers.  As explained in~\cite{Fos11}, this map is
well-defined, and it is a strictly increasing bijection of $\mathbb Q \cup
\nobreak \{\pm \infty\}$ onto $\mathbb Z[\frac 12] \cap \nobreak [-\frac 12,
\frac 12]$.  Therefore, it has a unique extension to a strictly increasing
homeomorphism, also called $\psi$, of $[-\infty, \infty]$ onto $[-\frac 12,
\frac 12]$.

\begin{remark}
                                                          \label{questionmark}
    Let $?\colon [0,1] \to [0,1]$ denote the Minkowski question mark function
    (see, for example,~\cite{Salem43}).  Then $?$ is constructed recursively
    in the same way as $\psi$, but on the interval~$[0,1]$.  However, one
    starts with the assignments
    \begin{center}
        $?(0) = 0$ \quad and \quad $?(1) = 1$,
        \end{center}
    while in the case of $\psi$ (on the interval $[0,1]$) we have
    \begin{center}
        $\psi(0) = 0$ \quad and \quad $\psi(1) = \frac 14$.
        \end{center}
    Therefore $\psi(x) = \frac 14 ?(x)$, for $0 \leq x \leq 1$.  Moreover, it
    is not hard to check that $\psi$ satisfies the following symmetries
        $\psi(-x) = -\psi(x)$ and
        $\psi(\tfrac 1x) = \tfrac 12 - \psi(x)$,
    for $x \in [0,\infty]$.  Hence, $\psi$ can be expressed in terms of the
    Minkowski question mark function by
    \begin{equation*}
        \psi(x) = \begin{cases}
            -\frac 12 + \frac 14?(-\frac 1x) & x\in [-\infty, -1] \\
            -\frac 14?(-x) & x\in [-1, 0]\\
            \frac 14?(x) & x\in [0, 1]\\
            \frac 12 - \frac 14?(\frac 1x) & x\in [1, \infty]
            \end{cases}.
        \end{equation*}

    Thurston's version of Theorem~\ref{thm:thurston} states that $T$ is
    isomorphic to the group of homeomorphisms of~$[0,1]$ with the endpoints
    identified, which are piecewise in $\PSL(2, \mathbb Z)$, have only
    finitely many breakpoints, all of which being in~$\mathbb Q \cap [0,1]$.
    This isomorphism is realized by conjugating by the Minkowski question mark
    function $?$, instead of $\psi$.
    \end{remark}

It is well-known (see~\cite[Example~1.5.3]{SerT}) that $\PSL(2, \mathbb Z)$ is
isomorphic to $\mathbb Z_2 \ast \mathbb Z_3 = \langle a,b \mid a^2 = b^3 = \I
\rangle$, with generators
\begin{equation*}
    a = \begin{bmatrix} 0 & -1 \\ 1 & 0\end{bmatrix} \{\pm\I\}
    \qquad \text{and} \qquad
    b = \begin{bmatrix} 0 & -1 \\ 1 & 1\end{bmatrix} \{\pm\I\}.
    \end{equation*}
Note that $\PSL(2, \mathbb Z)$ is a subgroup of $\PPSL(2, \mathbb Z)$.  Hence
by~\cite[Remark~2.3]{Fos11}, we have the following result.

\begin{proposition}
    Let $\Phi \colon \PPSL(2, \mathbb Z) \to T$ be the isomorphism from
    Theorem~\ref{thm:thurston}, and let $\Lambda$ denote the
    subgroup~$\Phi(\PSL(2, \mathbb Z))$ of $T$.  Then
    \begin{equation*}
        D^2 = CA = \Phi(a) \qquad \text{and} \qquad C = \Phi(b)
        \end{equation*}
    are free generators of~$\Lambda \cong \mathbb Z_2 \ast \mathbb Z_3$, of
    order~2 and~3, respectively.
    \end{proposition}

We will also need the following well-known (and easy to prove) fact about the
action of the group~$\PSL(2,\mathbb Z)$ on~$\mathbb R \cup \nobreak
\{\infty\}$ by M\"{o}bius transformations.

\begin{proposition}
                                                            \label{Q:uniquely}
    An element~$f \in \PSL(2,\mathbb Z)$ is uniquely determined by its value
    on two distinct points in~$\mathbb Q \cup \nobreak \{\infty\}$.
    \end{proposition}

\begin{remark}
                                                       \label{dyadic:uniquely}
    Note that we can translate the above result to one about~$\Lambda$, by
    using the homeomorphism $\phi$.  Indeed, $\phi$ restricts to a bijection
    from~$\mathbb Q \cup \nobreak \{\infty\}$ to~$[0,1) \cap \nobreak \mathbb
    Z[\frac 12]$, so an element in~$\Lambda$ is uniquely determined by its
    value on two points in~$[0,1) \cap \nobreak \mathbb Z[\frac 12]$.
    \end{remark}

\section{Non-inner amenability of \texorpdfstring{$T$}{T} and
\texorpdfstring{$V$}{V}}
                                                               \label{sec:non}

In this section we prove that the Thompson groups $T$ and $V$ are not inner
amenable.  This answers a question that Ionut Chifan raised at a conference in
Alba-Iulia in~2013.  We start by recalling the fundamentals of inner
amenability.

The \emph{reduced group $C^\ast$\mbox-algebra} associated to a discrete group
$G$, denoted by~$C^\ast_r(G)$, is the $C^\ast$\mbox-algebra generated by the
image of the left regular representation of $G$, $\lambda \colon G \to
B(\ell^2(G))$, while the \emph{group von Neumann algebra} associated to $G$,
denoted by $\LL G$, is the weak operator closure of $C^\ast_r(G)$ inside
$B(\ell^2(G))$.  The group von Neumann algebra is a type
$\mathrm{II}_1$\mbox-factor if and only if the group is ICC.

The von Neumann algebra $\LL G$ of a discrete ICC group $G$ is said to have
\emph{property~$\Gamma$}, if there exists a net $(u_i)_{i \in I}$ of unitaries
in $\LL G$ such that $\langle u_i \delta_e \mid \delta_e \rangle = 0$, for all
$i \in I$, and $\lim_{i \in I} \| (u_i x - x u_i)\delta_e \| = 0$, for all
$x\in \LL G$.  Property~$\Gamma$ can be defined for general finite factors,
but we are only concerned here with the case where the von Neumann algebra is
a group von Neumann algebra.

Recall that a discrete group~$G$ is said to be \emph{inner amenable}, if there
exists a state~$m$ on~$\ell^\infty(G)$ such that $m(\delta_e) = 0$ and $m(f) =
m(\alpha(g)f)$, for all $f \in \ell^\infty(G)$ and $g \in G$, where
$(\alpha(g)f)(h) = f(g^{-1} h g)$, for all $h \in G$.  Inner amenability was
introduced by Effros in 1975 in an attempt to characterize property~$\Gamma$
for group von Neumann algebras in terms of a purely group theoretic property;
see~\cite{Eff}.  Therein he proved the following theorem.

\begin{theorem}[Effros]
                                                            \label{thm:effros}
    For a discrete ICC group~$G$, the following statements are equivalent:
    \begin{enumerate}[label=(\roman{enumi})]

        \item
        $G$ is inner amenable.

        \item
        There exists a net~$(\eta_i)_{i\in I}$ of unit vectors in~$\ell^1(G)$
        satisfying $\eta_i(e) = 0$, for all $i \in I$, and $\lim_{i\in I} \|
        \alpha(g) \eta_i - \eta_i \|_1 = 0$, for all $g \in G$.

        \item
        There exists a net~$(\xi_i)_{i\in I}$ of unit vectors in~$\ell^2(G)$
        satisfying $\xi_i(e) = 0$, for all $i \in I$, and $\lim_{i\in I} \|
        \alpha(g) \xi_i - \xi_i \|_2 = 0$, for all $g \in G$.

        \end{enumerate}
    Moreover, if the group von Neumann algebra of~$G$ has property~$\Gamma$,
    then $G$ is inner amenable.
    \end{theorem}

Effros conjectured that inner amenability of a discrete ICC group~$G$ was, in
fact, equivalent to property~$\Gamma$ of~$\LL G$.  In~2012, Vaes~\cite{Vae12}
provided a counterexample.

Recall that an \emph{action} of a discrete group~$G$ on a set~$\mathfrak X$ is
a homomorphism from~$G$ to the set of permutations of~$\mathfrak X$, and that
such an action is said to be \emph{amenable} if there exists a finitely
additive probability measure on~$\mathfrak X$ which is invariant under the
action.  It straightforward to check that a discrete group~$G$ is inner
amenable if and only if the conjugation action of the group on $G\setminus
\{e\}$ is amenable.

A key result for our proof of non-inner amenability of~$T$ and~$V$ is the
following result due to Rosenblatt~\cite{Ros81}.

\begin{proposition}[Rosenblatt]
                                                        \label{thm:rosenblatt}
    Let $G$ be a non-amenable discrete group acting on a set~$\mathfrak X$.
    If the stabilizer of each point is amenable, then the action of $G$ is
    itself non-amenable.
    \end{proposition}

From this we directly get the following sufficient condition for non-inner
amenability.

\begin{corollary}
                                                               \label{cor:nia}
    Let $G$ be discrete a group.  If $G$ has a non-amenable subgroup~$H$ such
    that $\{ g \in H : ghg^{-1} = h \}$ is amenable, for each~$h \in G
    \setminus \{e\}$, then $G$ is non-inner amenable.
    \end{corollary}

Notice that, in the corollary above, any subgroup of~$G$ containing~$H$ will
automatically fail to be inner amenable, as well.

\begin{theorem}
    The Thompson groups $T$ and $V$ are not inner amenable.
    \end{theorem}
\begin{proof}
    Recall from Section~\ref{sec:iso} that $\Lambda$ denotes the subgroup
    of~$T$ generated by~$C$ and~$D^2$, and that conjugation by the map~$\phi$
    from Theorem~\ref{thm:thurston} restricts to an isomorphisms
    between~$\Lambda$ and~$\PSL(2, \mathbb Z)$.  Clearly $\Lambda$ is not
    amenable, so, by Corollary~\ref{cor:nia}, it suffices to prove that $\{g
    \in \Lambda : gf = fg\}$ is amenable, for all~$f \in V \setminus \{e\}$.
    This will imply that both~$T$ and~$V$ are non-inner amenable.  We will
    consider separately the cases where $f \in \Lambda \setminus \{e\}$, $f
    \in T \setminus \Lambda$ and $f \in V \setminus T$, respectively.  Fix~$f
    \in V \setminus \{e\}$ and let us denote the subgroup $\{g \in \Lambda :
    gf = fg\}$ by~$H$.

    First, suppose that $f \in \Lambda \setminus \{e\}$.  Then $H$ is the
    centralizer of~$f$ in~$\Lambda$, which is cyclic by Theorems~2.3.3
    and~2.3.5 in~\cite{Kat92}.  In particular, it is amenable.

    Now, suppose that $f \in T \setminus \Lambda$.  It is easy to see that,
    with $\tilde f$ denoting $\phi^{-1}f\phi$, we have $\phi^{-1} H \phi = \{
    h \in \PSL(2, \mathbb Z) : h \tilde f) = \tilde f h \}$.  Let us show that
    this group is amenable, since then $H$ will be amenable as well.  Since $f
    \notin \Lambda$, we know that $\tilde f \notin \PSL(2, \mathbb Z)$, and,
    in particular, $\tilde f$ has at least two breakpoints.  Let $y_1, \ldots,
    y_n$ denote these breakpoints, and note that they all are in~$\mathbb Q
    \cup \{\infty\}$.  Suppose that $h \in \PSL(2, \mathbb Z)$ with $h \tilde
    f h^{-1} = \tilde f$.  Clearly, the breakpoints of~$h \tilde f h^{-1}$ are
    $h(y_1), \ldots, h(y_n)$, so since $h \tilde f h^{-1} = \tilde f$, these
    must also be breakpoints of~$\phi^{-1}f\phi$.  In other words, $h$
    permutes the breakpoints of~$\tilde f$.  Now, because $f$ has at least two
    breakpoints, it follows from Proposition~\ref{Q:uniquely} that $h$ is
    uniquely determined by the corresponding permutation of $y_1, \ldots,
    y_n$.  Since there are only finitely many permutations of $n$~elements, we
    deduce that $\phi^{-1} H \psi$ is finite.  In particular, $H$ is also
    finite, and therefore amenable.

    Last, suppose that $f \in V \setminus T$.  Then $f$ is discontinuous, and
    since $f$ is a bijection of~$[0,1)$, it is easy to see that $f$ must have
    at least two points of discontinuity.  Since these points of discontinuity
    are all dyadic rationals and the elements of~$\Lambda$ are uniquely
    determined by their values of two dyadic rational numbers, by
    Remark~\ref{dyadic:uniquely}, we can use the same argument as the one
    for~$T \setminus \Lambda$ to conclude that $H$ is finite, since every
    element in~$\Lambda$ must permute points of discontinuity of~$f$.  This
    proves that $\{g \in \Lambda : gf = fg\}$ is amenable, for all~$f \in V
    \setminus \{e\}$.  By Corollary~\ref{cor:nia}, the conclusion follows.
    \end{proof}

The proof of the above theorem shows that, in fact, every subgroup of $V$
containing $\Lambda$ is non-inner amenable.

By the results of Effros (see Theorem~\ref{thm:effros}), it follows that
neither~$\LL T$ nor~$\LL V$ has property~$\Gamma$, and hence they are not
McDuff factors.

\section{Simplicity of \texorpdfstring{$C_r^\ast(T)$}{C*r(T)} implies
non-amenability of \texorpdfstring{$F$}{F}}
                                                        \label{sec:simplicity}

In this section we prove that the Thompson group~$F$ is non-amenable if the
reduced group $C^\ast$\mbox-algebra of~$T$ is simple.  First, we recall the
notion of weak containment of group representations.

By a representation of a group, we always mean a unitary representation, that
is, a homomorphism from the given group to the unitary group of some Hilbert
space.  We denote the full group $C^\ast$\mbox-algebra of a discrete group~$G$
by $C^\ast(G)$.  If $\pi \colon G \to B(H)$ is a representation of~$G$ on a
Hilbert space~$H$, then it extends uniquely to a representation
of~$C^\ast(G)$.  We use the same symbol to denote this representation.  If
$\pi \colon G \to B(H)$ and $\rho \colon G \to B(K)$ are two representations
of the group~$G$, then we say that $\pi$ is \emph{weakly contained} in $\rho$
if, for every $\xi \in H$ and $\varepsilon > 0$, there exist $\eta_1, \ldots,
\eta_n \in K$ such that
\begin{align*}
    \Bigl| \langle \pi(g) \xi \mid \xi \rangle - \sum_{i=1}^n \langle \rho(g)
    \eta_i \mid \eta_i \rangle \Bigr| < \varepsilon.
    \end{align*}
This is equivalent to the fact that there exists a $\ast$\mbox-homomor\-phism
$\sigma$ from the $C^\ast$\mbox-algebra generated by $\rho(G)$ to the
$C^\ast$\mbox-algebra generated by $\pi(G)$ such that $\pi = \sigma \circ
\rho$.  We need the following well-known result about weak containment.

\begin{proposition}
                                                               \label{pro:awc}
    Let $G$ be a discrete group acting on a set $\mathfrak X$.  If all
    stabilizers of the action are amenable, then the representation of~$G$
    on~$\ell^2(\mathfrak X)$ is weakly contained in the left regular
    representation of~$G$.
    \end{proposition}

For more about weak containment and connections to amenability see~\cite{BHV}.
The following result is a well-known characterization of amenable actions.
For a proof in the case where the action is left translation of the group on
itself, the reader may consult \cite[Theorem 2.6.8]{BO}.

\begin{proposition}
                                                             \label{ac,am,gen}
    Suppose that $G$ is a discrete group acting on a set~$\mathfrak X$, and
    let $\pi$ denote the induced representation on~$\ell^2(\mathfrak X)$.
    Then the action of~$G$ on~$\mathfrak X$ is non-amenable if and only if
    there exist $g_1, \ldots, g_n$ in~$G$ so that
    \begin{equation*}
        \Big\| \frac{1}{n} \sum_{k=1}^n \pi(g_k)\Big\| < 1.
        \end{equation*}
    In fact, if the action is non-amenable and $G$ is finitely generated, then
    $\| \frac{1}{n} \sum_{k=1}^n \pi(g_k)\| < 1$, for any set of elements
    $g_1, g_2, \ldots, g_n$ generating $G$.
    \end{proposition}

In the following we will make use of the \emph{Cuntz algebra $\mathcal O_2$}.
Recall that $\mathcal O_2$ is the universal $C^\ast$\mbox-algebra generated by
two isometries~$s_1$ and~$s_2$ with orthogonal range projections summing up to
the identity.  In other words, $\mathcal O_2$ is the universal
$C^\ast$\mbox-algebra generated by elements~$s_1$ and~$s_2$ satisfying the
relations $s_1^\ast s_1 = \I$, $s_2^\ast s_2 = \I$ and $s_1 s_1^\ast + s_2
s_2^\ast = \I$.  There is a canonical way to realize the Thompson groups as
subgroups of the unitary group of~$\mathcal O_2$.  This was discovered by
Nekrashevych in~\cite{Nek04}, as kindly pointed out to us by Wojciech
Szymanski.

Let us describe the concrete model we will use for the Cuntz algebra $\mathcal
O_2$.  We think of it as a specific set of bounded operators on
$\ell^2(\mathfrak X)$ where $\mathfrak X$~denotes the set $\mathbb Z[\tfrac
12] \cap \nobreak [0,1)$.  Define the operators~$s_1$ and~$s_2$
on~$\ell^2(\mathfrak X)$ by
\begin{align*}
    s_1 \delta_x = \delta_{x/2}
    \qquad \text{and} \qquad
    s_2 \delta_x = \delta_{(1 + x)/2},
    \end{align*}
for all $x \in \mathfrak X$.  It is straightforward to check that $s_1$
and~$s_2$ are isometries satisfying $s_1s_1^\ast + s_2 s_2^\ast = \I$, so they
generate a copy of the Cuntz algebra~$\mathcal O_2$.

The groups $F$,~$T$ and~$V$ act by definition on the set~$\mathfrak X$, and we
denote the induced representations on $\mathbb B(\ell^2(\mathfrak X))$
by~$\pi$.  That is, $\pi(g)\delta_x = \delta_{g(x)}$, for all $x \in \mathfrak
X$ and $g \in V$.  We use $\pi$ to denote this representation when restricted
to $F$ and~$T$, as well.  As it so happens, the image of~$\pi$ is contained
in~$\mathcal O_2$.  In fact, one can check the following explicit identities:
\begin{gather*}
    \pi(D)      =   s_2 s_2 s_1^\ast s_1^\ast
                +   s_1 s_1 s_2^\ast s_1^\ast
                +   s_1 s_2 s_1^\ast s_2^\ast
                +   s_2 s_1 s_2^\ast s_2^\ast, \\
    \pi(C)      =   s_2 s_2 s_1^\ast
                +   s_1 s_1^\ast s_2^\ast
                +   s_2 s_1 s_2^\ast s_2^\ast, \qquad
    \pi(D^2)    =   s_2 s_1^\ast
                    + s_1 s_2^\ast, \\
    \pi(\pi_0)  =   s_2 s_1 s_1^\ast
                + s_1 s_1^\ast s_2^\ast
                + s_2 s_2 s_2^\ast s_2^\ast.
    \end{gather*}

We denote the $C^\ast$\mbox-algebras generated by $\pi(F)$,~$\pi(T)$
and~$\pi(V)$ inside~$\mathcal O_2$ by $C^\ast_\pi(F)$,~$C^\ast_\pi(T)$
and~$C^\ast_\pi(V)$, respectively.

\begin{proposition}
    With the notation above, we have
    \begin{equation*}
        C^\ast_\pi(F)
            \subsetneq C^\ast_\pi(T)
            \subsetneq C^\ast_\pi(V)
            = \mathcal O_2.
        \end{equation*}
    \end{proposition}
\begin{proof}
    By construction, $C^\ast_\pi(F) \subseteq C^\ast_\pi(T) \subseteq
    C^\ast_\pi(V) \subseteq \mathcal O_2$.  We will prove that the first two
    inclusions are proper and that the last one is an equality.

    It is easy to see that $C^\ast_\pi(F) \neq C^\ast_\pi(T)$, since $\mathbb
    C\delta_0$ is a $C^\ast_\pi(F)$-invariant subspace which is not
    $C^\ast_\pi(T)$-invariant.  Hence $C^\ast_\pi(F) \neq C^\ast_\pi(T)$.

    For the rest, let us first prove that $C^\ast_\pi(V) = \mathcal O_2$, and
    afterwards that $C^\ast_\pi(T) \neq \mathcal O_2$.  Our strategy for
    showing that $C^\ast_\pi(V) = \mathcal O_2$ is to prove that
    $C^\ast_\pi(V)$ contains~$s_1$ and~$s_2$.  In fact, since $s_2 = \pi(D^2)
    s_1$, it suffices to show that it contains $s_1$.  To ease notation, let
    us denote $\ell^2(\mathbb Z[\frac 12] \cap \nobreak [0,\frac 12))$ and
    $\ell^2(\mathbb Z[\frac 12]\cap \nobreak [\frac 12, 1))$ by $\mathcal H_1$
    and~$\mathcal H_2$, respectively.  Consider the subgroup~$V_0$ of~$V$
    consisting of elements~$g \in V$ satisfying $g(x) = x$, for all $x \in
    [0,\frac 12)$.  Note that, for $g\in V_0$, both $\mathcal H_1$
    and~$\mathcal H_2$ are invariant subspaces for $\pi(g)$, so that we can
    write $\pi(g) = \pi(g)|_{\mathcal H_1} \oplus \pi(g)|_{\mathcal H_2}$.  It
    is easy to see using Theorem~\ref{thm:rosenblatt} that the action of~$V_0$
    on~$[\frac 12, 1)$ is non-amenable.  Thus, by Proposition~\ref{ac,am,gen},
    there exist elements $g_1, \ldots, g_n$ in~$V_0$, such that $\| \frac 1n
    \sum_{k=1}^n \pi(g_k)|_{\mathcal H_2} \| < 1$.  Since $\frac 1n
    \sum_{k=1}^n \pi(g_k)|_{\mathcal H_1} = \I_{\mathcal H_1}$, we deduce that
    $(\frac 1n \sum_{k=1}^n \pi(g_k))^m$ converges in norm, as~$m \to \infty$,
    to the projection of~$\ell^2(\mathfrak X)$ onto~$\mathcal H_1$, that is,
    to $s_1 s_1^\ast$.  Hence $s_1 s_1^\ast$, and therefore also $s_2 s_2^\ast
    = \I - s_1 s_1^\ast$, belong to $C_\pi^\ast(V)$.  It is straightforward to
    check that
    \begin{equation*}
        s_1 = \pi(A) s_1 s_1^\ast + \pi(D^2A^{-1}) s_2 s_2^\ast.
        \end{equation*}
    This implies that $s_1 \in C^\ast_\pi(V)$, and we conclude that
    $C^\ast_\pi(V) = \mathcal O_2$.

    Last, let us prove that $C^\ast_\pi(T) \neq \mathcal O_2$.  We do this by
    exhibiting two different states~$\phi_0$ and~$\phi_1$ on~$\mathcal O_2$
    which agree on $C^\ast_\pi(T)$.  Let $\phi_0$ denote the vector state
    given by $\delta_0$.  It is well-known that elements of the form $s_{i_1}
    s_{i_2} \cdots s_{i_n} s_{i_1}^\ast s_{i_2}^\ast \cdots s_{i_m}^\ast$ span
    a dense subspace of $\mathcal O_2$.  Moreover, it is easy to check that
    $\phi_0(s_1^n(s_1^\ast)^m) = 1$, for all $n,m \geq 0$, while $\phi_0$ is
    zero on the rest of these elements.  Let $\phi_1$ denote the composition
    of~$\phi_0$ with the automorphism of~$\mathcal O_2$ that interchanges
    $s_1$ and~$s_2$.  Then, for all $n,m \geq 0$, $\phi_1$ satisfies
    $\phi_1(s_{i_1} \cdots s_{i_n} s_{i_1}^\ast \cdots s_{i_m}^\ast) = 0$,
    unless $i_1 = \ldots = i_n = 2$ and $j_1 = \ldots = j_m = 2$, in which
    case $\phi_1(s_2^n(s_2^\ast)^m) = 1$.  It follows easily from
    Nekrashevych's description of the representation~$\pi$
    (see~\cite[Section~9]{Nek04}) that for all~$g \in V$, either
    $\phi_k(\pi(g)) = 0$ or $\phi_k(\pi(g)) = 1$, where $k = 0,1$.  Similarly,
    it is easily seen that, for~$g \in V$, $\phi_0(\pi(g)) = 1$ if and only if
    $\lim_{x \to 0} g(x) = 0$, and that $\phi_1(\pi(g)) = 1$ if and only if
    $\lim_{x \to 1} g(x) = 1$.  In particular, we see that $\phi_k(\pi(g)) =
    \I_{F}(g)$, for all $g \in T$ and $k = 0,1$.  Thus $\phi_0$ and~$\phi_1$
    agree on~$\pi(T)$, so clearly they also agree on~$C^\ast_\pi(T)$.  It
    follows that $C^\ast_\pi(T) \neq \mathcal O_2$, as the two states clearly
    do not agree on~$\mathcal O_2$.
    \end{proof}

It turns out that the representation~$\pi$ is closely connected to amenability
of the Thompson group~$F$, as we show in the following.

\begin{proposition}
                                                            \label{prop:sigma}
    The Thompson group~$F$ is amenable if and only if $\pi$ is contained in
    the left regular representation of~$T$.  This is further equivalent to the
    fact that $\pi$ is contained in the left regular representation of~$F$.
    \end{proposition}
\begin{proof}
    Clearly, $\pi|_F$ is weakly contained in the left regular representation
    of~$F$ if $\pi$ is weakly contained in the left regular representation
    of~$T$.  Suppose now that $F$ is amenable.  Since $T$ acts transitively
    on~$\mathfrak X$, all the stabilizers of the action are isomorphic.  It
    follows from Proposition~\ref{pro:awc} that $\pi$ is weakly contained in
    the left regular representation of~$T$, as $F$ is the stabilizer of~$0$.

    Suppose that $\pi$ is weakly contained in the left regular representation
    of~$F$.  Let $p$ denote the projection onto $\mathbb C \delta_0$, which is
    a $\pi(F)$-invariant subspace of~$\ell^2(\mathfrak X)$.  Then $p\pi p$ is
    the trivial representation of~$F$.  This is contained in~$\pi$, so by
    transitivity of weak containment, we deduce that the trivial
    representation of~$F$ is weakly contained in the left regular
    representation of~$F$.  By~\cite[Theorem~G.3.2]{BHV}, this is equivalent
    to amenability of $F$.
    \end{proof}

\begin{theorem}
                                                        \label{thm:simplicity}
    If $C^\ast_r(T)$ is simple, then $F$ is non-amenable.
    \end{theorem}
\begin{proof}
    Suppose that $F$ is amenable, and let us then prove that $C^\ast_r(T)$ is
    not simple.  By Proposition~\ref{prop:sigma}, $\pi$ is weakly contained in
    the left regular representation of~$T$, so there exists a
    $\ast$\mbox-homomorphism~$\sigma$ from $C^\ast_r(T)$ to $C^\ast_\pi(T)$
    such that $\pi = \sigma \circ \lambda$.  Our goal is to show that the
    kernel of $\sigma$ is a non-trivial ideal in $C^\ast_r(T)$.  Since
    $\sigma$ is obviously not the zero map, we only need to show that its
    kernel contains a non-trivial element.  The $\ast$-homomorphism $\lambda$
    is injective on the complex group algebra $\mathbb CT$, so it suffices to
    find an element $x\neq 0$ in $\mathbb CT$ such that $\pi(x) = 0$.

    Consider the elements $a$ and $b$ of $T$ (and their product) defined by
    \begin{center}
    \hfill
    \begin{tikzpicture}[scale=.4]
        
        \draw [thick] (0,0) rectangle (8,8);
        \node at (4,-1) {$a = CDC$};
        
        \draw [help lines] (1,0) -- (1,8);  \draw [help lines] (0,1) -- (8,1);
        \draw [help lines] (2,0) -- (2,8);  \draw [help lines] (0,2) -- (8,2);
        \draw [help lines] (3,0) -- (3,8);  \draw [help lines] (0,3) -- (8,3);
        \draw [help lines] (4,0) -- (4,8);  \draw [help lines] (0,4) -- (8,4);
        \draw [help lines] (5,0) -- (5,8);  \draw [help lines] (0,5) -- (8,5);
        \draw [help lines] (6,0) -- (6,8);  \draw [help lines] (0,6) -- (8,6);
        \draw [help lines] (7,0) -- (7,8);  \draw [help lines] (0,7) -- (8,7);

        \draw [very thick] (0,0) -- (4,4) -- (5,6) -- (6,7) -- (8,8);

        \end{tikzpicture}
    \hfill
    \begin{tikzpicture}[scale=.4]
        
        \draw [thick] (0,0) rectangle (8,8);
        \node at (4,-1) {$b = D^2CDCD^2$};
        
        \draw [help lines] (1,0) -- (1,8);  \draw [help lines] (0,1) -- (8,1);
        \draw [help lines] (2,0) -- (2,8);  \draw [help lines] (0,2) -- (8,2);
        \draw [help lines] (3,0) -- (3,8);  \draw [help lines] (0,3) -- (8,3);
        \draw [help lines] (4,0) -- (4,8);  \draw [help lines] (0,4) -- (8,4);
        \draw [help lines] (5,0) -- (5,8);  \draw [help lines] (0,5) -- (8,5);
        \draw [help lines] (6,0) -- (6,8);  \draw [help lines] (0,6) -- (8,6);
        \draw [help lines] (7,0) -- (7,8);  \draw [help lines] (0,7) -- (8,7);

        \draw [very thick] (0,0) -- (1,2) -- (2,3) -- (4,4) -- (8,8);

        \end{tikzpicture}
    \hfill
    \begin{tikzpicture}[scale=.4]
        
        \draw [thick] (0,0) rectangle (8,8);
        \node at (4,-1) {$ab$ $(= ba)$};
        
        \draw [help lines] (1,0) -- (1,8);  \draw [help lines] (0,1) -- (8,1);
        \draw [help lines] (2,0) -- (2,8);  \draw [help lines] (0,2) -- (8,2);
        \draw [help lines] (3,0) -- (3,8);  \draw [help lines] (0,3) -- (8,3);
        \draw [help lines] (4,0) -- (4,8);  \draw [help lines] (0,4) -- (8,4);
        \draw [help lines] (5,0) -- (5,8);  \draw [help lines] (0,5) -- (8,5);
        \draw [help lines] (6,0) -- (6,8);  \draw [help lines] (0,6) -- (8,6);
        \draw [help lines] (7,0) -- (7,8);  \draw [help lines] (0,7) -- (8,7);

        \draw [very thick] (0,0)--(1,2)--(2,3)--(4,4)--(5,6)--(6,7)--(8,8);

        \end{tikzpicture}
    \hfill \mbox{}
    \end{center}
    Obviously, $a$ and $b$ commute, and it is easy to see that the values
    $(a(x), b(x))$ and $(ab(x), x)$ agree, up to a permutation, for all $x \in
    [0,1)$.  From this we deduce that $\pi(a + b)\delta_x = \pi(ab +
    e)\delta_x$, for all $x \in \mathfrak X$.   In particular, $\pi(a+b-ab-e)
    = 0$, so that the kernel of $\sigma$ is non-trivial.  Thus $C^\ast_r(T)$
    is not simple.
    \end{proof}

As mentioned in the introduction, Le~Boudec and Matte Bon~\cite{LBMB16}
recently proved the converse implication to the theorem above.  They also
proved therein that $C^\ast_r(V)$ is simple, and that $C^\ast_r(F)$ is simple
if $C^\ast_r(T)$ is simple.  The converse of this latter implication was
already proved by Breuillard, Kalantar, Kennedy and Ozawa~\cite{BKKO14}.

It is still unknown whether the reduced group $C^\ast$-algebras of~$T$ is
simple, but it is, however, known that it has a unique tracial state.  This
was proved by Dudko and Medynets~\cite{DuMe12}.

\begin{remark}
                                                               \label{trace;F}
    We believe that is is a well-known fact that $F$ is non-amenable if and
    only if its reduced group $C^\ast$-algebra has a unique tracial state,
    although we have not been able to find an explicit reference for this.  It
    follows, for example, from the fact that $C^\ast_r(F)$ has at most two
    extremal tracial states, as showed by Dudko and Medynets~\cite{DuMe12}.  A
    different proof of this is based on the new characterization of groups
    whose reduced group $C^\ast$-algebra has a unique tracial state given by
    Breuillard, Kalantar, Kennedy and Ozawa.  In the following we briefly
    indicate the argument.

    It was shown in the aforementioned paper that $C^\ast_r(F)$ has a unique
    tracial state if and only if the amenable radical of~$F$ is trivial.
    However, it is a well-known fact (see~\cite[Theorem~4.3]{CFP96}) that
    every non-trivial normal subgroup of $F$ contains a copy of~$F$.  Hence
    the amenable radical is trivial if and only if $F$ is non-amenable.
    \end{remark}

We end this section by giving a few new equivalent characterizations of
amenability of~$F$.  We will make use of the following well-known result, for
which a proof can be found in~\cite[Lemma~4.1]{Haa15} (in the case where the
action is left translation of the group on itself).

\begin{lemma}
                                                         \label{positive;norm}
    Let $G$ be a discrete group acting on a set~$\mathfrak X$, and
    let~$\sigma$ be the corresponding representation on~$\ell^2(\mathfrak X)$.
    Then $\| x + y \| \geq \| x \|$, for $x,y \in \mathbb R_+ G$.
    \end{lemma}

By scrutinizing the proof of Theorem~\ref{thm:simplicity} one sees that, with
the notation therein, the Thompson group~$F$ is non-amenable if the closed
two-sided ideal generated by the element $\I + \lambda(ab) - \lambda(a) -
\lambda(b)$ inside~$C^\ast_r(T)$ is the whole of~$C^\ast_r(T)$.  This element
is not unique with this property.  Indeed, the same holds true for
$\lambda(x)$, where $x$ is an element of the complex group algebra $\mathbb C
T$ such that $\pi(x) = 0$.  Bleak and Juschenko~\cite{BJ14} established a
partial converse to Theorem~\ref{thm:simplicity} by proving that such~$x$
exists if the Thompson group~$F$ is non-amenable.  The equivalence of the
first two statements in the following proposition is a strengthening of their
result, in the sense that we exhibit a concrete such element~$x$.

\begin{proposition}
                                                                 \label{ideal}
    With $a$ and~$b$ as above, the following are equivalent:
    \begin{enumerate}[%
        label=\textup{(\arabic{enumi})},
        ref=(\arabic{enumi}),
        itemsep=2pt plus 1pt minus 1pt]

        \item
                                                            \label{ideal:amen}
        The Thompson group~$F$ is non-amenable.

        \item
                                                         \label{ideal:whole;T}
        The closed two-sided ideal generated by $\I + \lambda(ab) - \lambda(a)
        - \lambda(b)$ in $C^\ast_r(T)$ is all of $C^\ast_r(T)$.

        \item
                                                          \label{ideal:zero;T}
        The closed convex hull of $\{ \lambda(hah^{-1}) + \lambda(hbh^{-1}) :
        h \in T \}$ contains~$0$.

        \item
                                                         \label{ideal:whole;F}
        The closed two-sided ideal generated by $\I + \lambda(ab) - \lambda(a)
        - \lambda(b)$ in $C^\ast_r(F)$ is all of $C^\ast_r(F)$.

        \item
                                                          \label{ideal:zero;F}
        The closed convex hull of $\{ \lambda(hah^{-1}) + \lambda(hbh^{-1}) :
        h \in F \}$ contains~$0$.

        \end{enumerate}
    \end{proposition}
\begin{proof}
    We prove that \ref*{ideal:amen} implies~\ref*{ideal:whole;F}
    and~\ref*{ideal:zero;F}, that \ref*{ideal:whole;F}
    implies~\ref*{ideal:whole;T}, that \ref*{ideal:zero;F}
    implies~\ref*{ideal:zero;T}, that \ref*{ideal:whole;T}
    implies~\ref*{ideal:amen}, and that \ref*{ideal:zero;T}
    implies~\ref*{ideal:amen}.

    Some of these implications are straightforward.  That \ref*{ideal:whole;F}
    implies~\ref*{ideal:whole;T} follows from the fact that the inclusion
    $C^\ast_r(F) \subseteq C^\ast_r(T)$ is unital.  That is, if
    \ref*{ideal:whole;F} holds, then the closed two-sided ideal
    in~$C^\ast_r(T)$ generated by $\I + \lambda(ab) - \lambda(a) - \lambda(b)$
    contains~$C^\ast_r(F)$, and, in particular, the unit.  That
    \ref*{ideal:zero;F} implies~\ref*{ideal:zero;T} follows directly from the
    fact that the inclusion~$C^\ast_r(F) \subseteq \nobreak C^\ast_r(T)$ is
    isometric.

    Let us prove that \ref*{ideal:whole;T} implies~\ref*{ideal:amen}, and that
    \ref*{ideal:zero;T} implies~\ref*{ideal:amen}, or, equivalently, that the
    negation of~\ref*{ideal:amen} implies the negation
    of~\ref*{ideal:whole;T}, as well as the negation of~\ref*{ideal:zero;T}.
    Assume that $F$ is amenable.  Then by Proposition~\ref{prop:sigma}, $\pi$
    is weakly contained in the left regular representation of~$T$, that is,
    there exists a $\ast$\mbox-homomorphism~$\sigma$ so that $\pi = \sigma
    \circ \lambda$.  As explained in the proof of
    Theorem~\ref{thm:simplicity}, $\I + \lambda(ab) - \lambda(a) - \lambda(b)$
    is in the kernel of~$\sigma$, which is a non-trivial ideal.  Hence the
    negation of~\ref*{ideal:whole;T} holds.  Since $\sigma$ is a contraction,
    to prove the negation of~\ref*{ideal:zero;T} it suffices to show that $\|
    x \| \geq 1$, for every $x \in \cconv\{ \pi(hah^{-1}) + \pi(hbh^{-1}) : h
    \in T \}$.  As explained in the proof of Theorem~\ref{thm:simplicity},
    that $\pi(a) + \pi(b) = \I + \pi(ab)$.  Hence
    \begin{equation*}
        \cconv\{ \pi(hah^{-1}) + \pi(hbh^{-1}) : h \in T \}
            = \I + \cconv\{ \pi(habh^{-1}) : h \in T \}.
        \end{equation*}
    If $x$ is a finite convex combination of elements of the form~$habh^{-1}$
    with~$h \in T$, then $x \in \mathbb R_+ T$.  By Lemma~\ref{positive;norm},
    we conclude that $\| \I + x \| \geq \|\I\| = 1$.  Continuity then ensures
    that $\| \I + x \| \geq 1$, for every $x \in \cconv\{ \pi(hah^{-1}) +
    \pi(hbh^{-1}) : h \in T \}$.  Hence the negation of~\ref*{ideal:zero;T}
    holds.

    We are left to prove that \ref*{ideal:amen} implies~\ref*{ideal:whole;F}
    and~\ref*{ideal:zero;F}, so assume that $F$ is non-amenable.  To
    prove~\ref*{ideal:whole;F}, it suffices to show that the closed two-sided
    ideal in~$C^\ast_r(F)$ generated by $\I + \lambda(ab) - \lambda(a) -
    \lambda(b)$ contains~$\I$.  As it clearly contains the set
    \begin{align*}
        \cconv \big\{ \lambda(h)\big(&\I + \lambda(ab) - \lambda(a) -
        \lambda(b) \big) \lambda(h)^\ast : h \in F \big\}\\
        &= \I + \cconv \big\{ \lambda(habh^{-1}) - \lambda(hah^{-1}) -
        \lambda(hbh^{-1}) : h \in F \big\},
        \end{align*}
    it suffices to show that the closed convex hull on the right hand side
    contains~$0$.  Let $\varepsilon > 0$ be given.  As mentioned in
    Remark~\ref{trace;F}, $F$ has the unique trace property.  Let $F_1$ and
    $F_2$ denote the subgroups of $F$ consisting of elements~$f \in F$ so that
    $f(x) = x$, for all~$x \in [0, \frac 12)$, and $f(x) = x$, for all~$x \in
    [\frac 12, 1)$, respectively.  Note that $b \in F_1$ and~$a \in F_2$.
    Since $F_1$ and~$F_2$ are both isomorphic to~$F$, they both have the
    unique trace property.  It follows from \cite[Corollary~4.4]{Haa15} that
    there exist positive real numbers $s_1, \ldots, s_n, t_1, \ldots, t_m$
    with $\sum_{k=1}^n s_k = \sum_{k=1}^m t_k = 1$, and $g_1, \ldots, g_n \in
    F_1$ and $h_1, \ldots, h_m \in F_2$ so that
    \begin{equation*}
        \Big\| \sum_{k=1}^n s_k \lambda(g_k b g_k^{-1}) \Big\| <
        \varepsilon
        \qquad \text{and} \qquad
        \Big\| \sum_{k=1}^m t_k \lambda(h_k a h_k^{-1}) \Big\| <
        \varepsilon.
        \end{equation*}
    Let us, for simplicity of notation, denote the elements 
        $\sum_{k=1}^m t_k \lambda(h_k a h_k^{-1})$ and 
        $\sum_{k=1}^n s_k \lambda(g_k b g_k^{-1})$ 
        by $\tilde a$ and $\tilde b$, respectively.
    As the elements of the subgroups~$F_1$ and~$F_2$ commute, it is
    straightforward to check that
    \begin{equation*}
            \sum_{i=1}^n \sum_{j=1}^m s_i t_j
                \lambda(g_ih_j) \big(
                \lambda(a) + \lambda(b)
                \big) \lambda(g_ih_j)^\ast
        = \tilde a + \tilde b.
        \end{equation*}
    Since $\sum_{i=1}^n \sum_{j=1}^m s_i t_j = 1$, the left hand side belongs
    to the convex hull of $\{ \lambda(hah^{-1}) + \lambda(hbh^{-1}) : h \in F
    \}$, and since $\| \tilde a + \tilde b \| < 2 \varepsilon$, we conclude
    that the convex hull of $\{ \lambda(hah^{-1}) + \lambda(hbh^{-1}) : h \in
    F \}$ contains elements of arbitrarily small norm.  Thus
    \ref*{ideal:zero;F} holds.

    Using similar calculations, it is straightforward to check that
    \begin{equation*}
        \sum_{i=1}^n \sum_{j=1}^m s_i t_j
              \lambda(g_ih_j) \big(
              \lambda(ab)
            - \lambda(a)
            - \lambda(b)
            \big) \lambda(g_ih_j)^\ast
        = \tilde a \tilde b - \tilde a - \tilde b.
        \end{equation*}
    Again, the left hand side belongs to the convex hull of
    \begin{equation*}
        \{ \lambda(habh^{-1})-\lambda(hah^{-1})-\lambda(hbh^{-1}) : h\in F \}.
        \end{equation*}
    Since $\| \tilde a \tilde b - \tilde a - \tilde b \| < \varepsilon^2 -
    2\varepsilon$, we conclude that the closure of this convex hull
    contains~$0$, and as mentioned above, this means that the ideal generated
    by $\I + \lambda(ab) - \lambda(a) - \lambda(b)$ in~$C^\ast_r(F)$ is the
    whole of~$C^\ast_r(F)$.  This proves that \ref*{ideal:amen}
    implies~\ref*{ideal:whole;F} and~\ref*{ideal:zero;F}, which concludes the
    proof.
    \end{proof}

\noindent{\bfseries Acknowledgement.}  This work is part of the second named
author's Ph.D thesis at the University of Copenhagen under the supervision of
Magdalena Musat and Uffe Haagerup.  Kristian Olesen wishes to express his
sincere gratitude for their invaluable guidance and support throughout the
Ph.D period.  This manuscript was completed after the sudden passing away of
Uffe Haagerup.


\begin{thebibliography}{10}

\bibitem{BHV}
Bachir Bekka, Pierre de~la Harpe, and Alain Valette.
\newblock {\em Kazhdan's property ({T})}, volume~11 of {\em New Mathematical
  Monographs}.
\newblock Cambridge University Press, Cambridge, 2008.

\bibitem{BJ14}
Collin Bleak and Kate Juschenko.
\newblock Ideal structure of the {$C^*$}-algebra of {T}hompson group {$T$}.
\newblock {\em preprint}, 2014.
\newblock arXiv:1409.8099.

\bibitem{BKKO14}
Emmanuel Breuillard, Mehrdad Kalantar, Matthew Kennedy, and Narutaka Ozawa.
\newblock {$C^*$}-simplicity and the unique trace property for discrete groups.
\newblock {\em preprint}, 2014.
\newblock arXiv:1410.2518.

\bibitem{BO}
Nathanial~P. Brown and Narutaka Ozawa.
\newblock {\em {$C^*$}-algebras and finite-dimensional approximations},
  volume~88 of {\em Graduate Studies in Mathematics}.
\newblock American Mathematical Society, Providence, RI, 2008.

\bibitem{CFP96}
J.~W. Cannon, W.~J. Floyd, and W.~R. Parry.
\newblock Introductory notes on {R}ichard {T}hompson's groups.
\newblock {\em Enseign. Math. (2)}, 42(3-4):215--256, 1996.

\bibitem{CSS01}
Tullio~G. Ceccherini-Silberstein and Fabio Scarabotti.
\newblock Inner amenability of some groups of piecewise-linear homeomorphisms
  of the real line.
\newblock {\em J. Math. Sci. (New York)}, 106(4):3164--3167, 2001.
\newblock Pontryagin Conference, 8, Algebra (Moscow, 1998).

\bibitem{dlH07}
Pierre de~la Harpe.
\newblock On simplicity of reduced {$C^\ast$}-algebras of groups.
\newblock {\em Bull. Lond. Math. Soc.}, 39(1):1--26, 2007.

\bibitem{DuMe12}
Artem Dudko and Konstantin Medynets.
\newblock Finite factor representations of {H}igman-{T}hompson groups.
\newblock {\em preprint}, 2012.
\newblock arXiv:1212.1230.

\bibitem{Eff}
Edward~George Effros.
\newblock Property {$\Gamma $} and inner amenability.
\newblock {\em Proc. Amer. Math. Soc.}, 47:483--486, 1975.

\bibitem{Fos11}
Ariadna Fossas.
\newblock {$\mathrm{PSL}(2,\mathbb Z)$} as a non-distorted subgroup of
  {T}hompson's group {$T$}.
\newblock {\em Indiana Univ. Math. J.}, 60(6):1905--1925, 2011.

\bibitem{Haa15}
Uffe Haagerup.
\newblock A new look at {$C^*$}-simplicity and the unique trace property of a
  group.
\newblock {\em preprint}, 2015.
\newblock arXiv:1509.05880.

\bibitem{Imb97}
Michel Imbert.
\newblock Sur l'isomorphisme du groupe de {R}ichard {T}hompson avec le groupe
  de {P}tol\'em\'ee.
\newblock In {\em Geometric {G}alois actions, 2}, volume 243 of {\em London
  Math. Soc. Lecture Note Ser.}, pages 313--324. Cambridge Univ. Press,
  Cambridge, 1997.

\bibitem{IO16}
Nikolay~A. Ivanov and Tron Omland.
\newblock {$C^\ast$}-simplicity of free products with amalgamation and radical
  classes of groups.
\newblock {\em preprint}, 2016.
\newblock arXiv:1605.06395.

\bibitem{Jol97}
Paul Jolissaint.
\newblock Moyennabilit\'e int\'erieure du groupe {$F$} de {T}hompson.
\newblock {\em C. R. Acad. Sci. Paris S\'er. I Math.}, 325(1):61--64, 1997.

\bibitem{Jol98}
Paul Jolissaint.
\newblock Central sequences in the factor associated with {T}hompson's group
  {$F$}.
\newblock {\em Ann. Inst. Fourier (Grenoble)}, 48(4):1093--1106, 1998.

\bibitem{KK14}
Mehrdad Kalantar and Matthew Kennedy.
\newblock Boundaries of reduced {$C^\ast$}-algebras of discrete groups.
\newblock {\em preprint}, 2014.
\newblock arXiv:1405.4359.

\bibitem{Kat92}
Svetlana Katok.
\newblock {\em Fuchsian groups}.
\newblock Chicago Lectures in Mathematics. University of Chicago Press,
  Chicago, IL, 1992.

\bibitem{Bou15}
Adrien Le~Boudec.
\newblock Discrete groups that are not {$C^\ast$}-simple.
\newblock {\em preprint}, 2015.
\newblock arXiv:1507.03452.

\bibitem{LBMB16}
Adrien Le~Boudec and Nicolás Matte~Bon.
\newblock Subgroup dynamics and {$C^\ast$}-simplicity of groups of
  homeomorphisms.
\newblock {\em preprint}, 2016.
\newblock arXiv:1605.01651.

\bibitem{Nek04}
Volodymyr~V. Nekrashevych.
\newblock Cuntz-{P}imsner algebras of group actions.
\newblock {\em J. Operator Theory}, 52(2):223--249, 2004.

\bibitem{Ros81}
Joseph Rosenblatt.
\newblock Uniqueness of invariant means for measure-preserving transformations.
\newblock {\em Trans. Amer. Math. Soc.}, 265(2):623--636, 1981.

\bibitem{Salem43}
R.~Salem.
\newblock On some singular monotonic functions which are strictly increasing.
\newblock {\em Trans. Amer. Math. Soc.}, 53:427--439, 1943.

\bibitem{SerT}
Jean-Pierre Serre.
\newblock {\em Trees}.
\newblock Springer-Verlag, Berlin-New York, 1980.
\newblock Translated from the French by John Stillwell.

\bibitem{Vae12}
Stefaan Vaes.
\newblock An inner amenable group whose von {N}eumann algebra does not have
  property {G}amma.
\newblock {\em Acta Math.}, 208(2):389--394, 2012.

\end{thebibliography}
%

\end{document}